\documentclass[a4paper,12pt]{amsart}

\usepackage[english]{babel}

\usepackage{amsmath}
\usepackage[all]{xy}
\usepackage{amssymb}
\usepackage[active]{srcltx}
\usepackage{fullpage}
\usepackage{hyperref}
\usepackage{color}

\usepackage{enumerate}

\usepackage[alphabetic,
lite]{amsrefs} 

\usepackage{color, xcolor}

\renewcommand{\boldsymbol}{\mathbf}

\makeatletter
\newtheorem*{rep@theorem}{\rep@title}
\newcommand{\newreptheorem}[2]{%
\newenvironment{rep#1}[1]{%
 \def\rep@title{#2 \ref{##1}}%
 \begin{rep@theorem}}%
 {\end{rep@theorem}}}
\makeatother

\newtheorem*{theorem*}{Theorem}
\newtheorem{theorem}{Theorem}
\newtheorem{proposition}[theorem]{Proposition}
\newtheorem{lemma}[theorem]{Lemma}
\newtheorem{corollary}[theorem]{Corollary}
\newtheorem{definition}[theorem]{Definition}

\newreptheorem{theorem}{Theorem}

\theoremstyle{definition}
\newtheorem{remark}[theorem]{Remark}
\newtheorem{example}[theorem]{Example}

\newcommand{\ZZ}{\mathbb{Z}}

\newcommand{\PP}{\mathbb{P}}

\newcommand{\commentpigna}[1]{}

\DeclareMathOperator{\Ad}{Ad}
\DeclareMathOperator{\Alb}{Alb}
\DeclareMathOperator{\Aut}{Aut}

\DeclareMathOperator{\Id}{Id}
\DeclareMathOperator{\Sym}{Sym}
\DeclareMathOperator{\Pic}{Pic}

\title[Further quotients and Albanese morphisms]{Quotients of the square of a curve by a mixed action, further quotients and Albanese morphisms}
\author{Roberto Pignatelli}
\address{Dipartimento di Matematica,
	Universit\`a di Trento,
	via Sommarive 14,
	I-38123 Trento, Italy.}
\email{Roberto.Pignatelli@unitn.it}
\thanks{
The author is grateful to Fabrizio Catanese for inviting him in Bayreuth with
the ERC-2013-Advanced Grant - 340258- TADMICAMT and for the several enlightening conversations on this subject during his stay in May 2018 that allowed to considerably simplify several proofs. He is grateful to Nicola Cancian, Christian Glei\ss ner, Pietro Pirola and Francesco Polizzi for several fruitful conversations on the subject of this paper, and to Davide Frapporti for his careful reading of the first draft of this paper. The author thanks the anonymous reviewers whose comments helped improve and clarify this manuscript.
 He was partially  supported  by  the  project  PRIN 2015 Geometria delle 
variet\`a algebriche.
He is a  member  of  GNSAGA-INdAM
}

\begin{document}


\begin{abstract}
We study \emph{mixed surfaces}, the minimal resolution $S$ of the singularities of a quotient $(C \times C)/G$ of the {\em square} of a curve by a finite group of automorphisms that contains elements not preserving the factors. We study them through the \emph{further quotients} $(C \times C)/G'$ where $G' \supset G$.

As a first application we prove that  if the irregularity is at least $3$, then $S$ is also minimal. The result is sharp. 

The main result is a complete description of the Albanese morphism of $S$ through a well determined further quotient $(C \times C)/G'$ that is an \'etale cover of the symmetric square of a curve. In particular, if the irregularity of $S$ is at least $2$, then $S$ has maximal Albanese dimension.

We apply our result to all the {\em semi-isogenous} mixed surfaces of maximal Albanese dimension constructed by Cancian and Frapporti, relating them with the other constructions appearing in the literature of surfaces of general type having the same invariants.
\end{abstract}

\maketitle
%






\tableofcontents

\section*{Introduction} \label{sec:intro}
The quotients of a product of two curves by the action of a finite group have been in the last couple of decades (essentially from \cite{Ca00} on) a very fruitful   method to construct interesting surfaces, mainly of general type. 

The research considered first  the case of free actions. The quotient of a product of two curves of respective genera at least $2$ (a surface {\em isogenous to a higher product}) is a surface $S$ with ample canonical class $K_S$ of self intersection 
\begin{equation}\label{K2=8X}
K^2_S=8\chi({\mathcal O}_S),
\end{equation}
 and rather few surfaces like this were known before. This case is still producing very interesting constructions answering very old questions, see e.g. the very recent paper \cite{Cat17}.

Then the research considered the case of  {\em product-quotient surfaces}, where the freeness assumption is substituted by requiring that the group acts as $g(x,y)=(gx,gy)$ (the {\it unmixed case}). One usually assumes the action to be {\it minimal}, {\it i.e.} faithful on both factors: indeed every such quotient may be obtained by an unique minimal action.

 In this case  (\ref{K2=8X}) is substituted by the much less restrictive $K^2_S \leq 8\chi({\mathcal O}_S)$ (and $K_S^2 \neq 8\chi({\mathcal O}_S)-1$ at least in the irregular case, see \cite{PolNumIso}*{Theorem B}). 
 We have now computer programs (\cites{BCGP12,BP10,BP16}) able to construct all product-quotient surfaces with given values of the invariants $K^2$, $p_g$ and $q$, and  interesting surfaces with several different  values of the invariants  have been constructed, see  \cite{MYSURVEYONQE} and the references therein. 

In this paper we consider the other case, the {\em mixed} one, where the group exchanges the factors.

This case seems much more complicated and at the moment we do not have a comparably developed theory for it, unless we add some more conditions, as for the \emph{quasi-\'etale} surfaces in \cite{Fr13} and \cite{FP15} and the \emph{semi-isogenous} surfaces in \cite{CanFr15}. 

In the theory of the irregular surfaces of general type it is very important to be able to describe the Albanese morphism. For example most of the results in  \cite{pigpol} are a consequence of the description of the Albanese morphism of the  surfaces considered there. 

The main result of this paper is Theorem \ref{Albanese and generalized dihedral surfaces}, giving a precise description of the Albanese morphism of irregular mixed surfaces. 

We give here a simplified version of it. Recall that  the quotients by a mixed action have an unique {\em minimal realization} as $(C \times C)/G$ where the action of $G$ is {\em minimal}. This means that the index two subgroup $G^0$ of the elements preserving the factors acts faithfully on both curves. In this case it is usual to see $G^0$ as subgroup of $\Aut(C)$ by means of the action on the first factor.

\begin{theorem}\label{main}
Let $S$ be an irregular mixed surface, minimal resolution of a mixed quotient $r \colon S \rightarrow X=(C \times C)/G$ where $G \subset \Aut(C \times C)$  is a minimal action. 

Let $\Sigma \subset G^0$ be the set of the elements fixing at least one point of $C$.

Let $K \subset G^0$ be the smallest  subgroup of $G^0$ normal  in $G$ containing $\Sigma$, the commutator subgroup $[G^0,G^0]$ and $\{ g^2 | g \in G \setminus G^0\}$.

Then $\Alb(S)$ is isogenous to $ J(C/G^0)$ via an isogeny with kernel isomorphic to $G^0/K$.  Moreover,  if $q(S) \geq 2$, then $S$ is of maximal Albanese dimension and the degree of its Albanese morphism equals $|K|$. 
\end{theorem}

This shows that this mixed case is a natural place where to look for simple constructions of surfaces of {\it maximal Albanese dimension}, a class of irregular surfaces that has attracted the attention of several researchers in the last years mostly in relation to the famous Severi inequality recently proved in \cite{pardaInv}: $K^2_S \geq 4\chi({\mathcal O}_S)$ for minimal smooth complex surfaces S of maximal Albanese dimension. 

Theorem \ref{Albanese and generalized dihedral surfaces} is more precise than Theorem \ref{main}. It factors the Albanese morphism of $S$ through a {\em further quotient}, the natural map  $(C \times C)/G \rightarrow (C \times C)/G_K$ where $G_K$ is a group of automorphisms of $C \times C$ containing $G$ determined by $K$, such that $(C \times C)/G_K$ belongs to a class of surfaces with special properties and very simple Albanese morphism that we call {\em dihedral surfaces}, see Definition \ref{GDS}. Actually Theorem \ref{Albanese and generalized dihedral surfaces} gives a complete description of the Albanese morphism of $S$ also in the case $q(S)=1$.

The group $G_K$ is a subgroup of the group  $G^0(2) \cong \left( G^0 \times G^0 \right) \rtimes \ZZ_{/2\ZZ}$ of the automorphisms of $C \times C$ generated by the automorphisms of the form $(x,y)\mapsto (gx,hy)$ and by the exchange of the factors $(x,y) \mapsto (y,x)$. Precisely, it is the subgroup generated by $G$ and $K \times \{1\}$. Actually
one of the main ingredients of the proof is a detailed analysis of all the subgroups of $G^0(2)$ containing $G$, summarized in Theorem \ref{G_K} that we repeat here.

\begin{theorem}
For every subgroup $G'$ of $G^0(2)$ containing $G$ consider the intersection $G' \cap \left (G^0 \times \{1\}\right)=:K \times \{1\}$ of the elements acting trivially on the second factor.

Then $K$  is normal in $G$. 

Moreover $K$ determines $G'$ in the sense that for every subgroup $K$ of $G^0$ normal in $G$ there is a unique group $G_K$, $G \subset G_K \subset G^0(2)$, such that $G_K \cap \left (G^0 \times \{1\}\right) = K \times \{1\}$.
\end{theorem}

Our analysis produced one more interesting result, Theorem \ref{minimality}, that we rewrite here.

\begin{theorem}
Let $S$ be a surface birational to the quotient of the square of a curve by a mixed action
 with irregularity $q(S) \geq 3$. Then $S$ is minimal.
\end{theorem}
The result is sharp, as nonminimal examples with irregularity $0$, $1$ and $2$ have been constructed by \cite{CanFr15}.  We mention that for special actions one can get better results as in  \cite{FP15}*{Theorem 4.5}. 

As application we describe in detail in Section \ref{AlbaneseMorphisms} the Albanese morphism of all semi-isogenous mixed surfaces $S$ with $p_g(S)=q(S)=2$ and $K_S^2>0$ , classified in \cite{CanFr15}. We collect in the following statement a small selection of the results in Section \ref{AlbaneseMorphisms}.

\begin{theorem}Let $S$ be a mixed semi-isogenous surface with $p_g=q=2$ and $K^2>0$. Then
\begin{itemize}
\item[-] the Albanese morphism of $S$  has degree $2$ or $3$;
\item[-] its ramification is simple;
\item[-]  $\Alb(S)$ is isogenous to the Jacobian of the curve $C/G^0$ of genus $2$ and carries a polarization that is either principal or of type $(1,2)$;
\item[-] if $K^2_S=6$, then $S$ is either one of the surfaces of type Ia or one of the surfaces of type Ib in \cite{PePol13b}, and both cases occur;
\item[-] the surfaces of type Ia  in \cite{PePol13b} are topologically distinct from those of type Ib;
\item[-] if $K^2_S=4$, then the minimal model of $S$ is one of the ``Chen-Hacon surfaces'' in \cite{PePol13a}.
\end{itemize}
\end{theorem}

It is worth mentioning that in all cases we find a finite map from $S$  to $J(C/G^0)$ simply branched on the image of the diagonal of $\Sym^2(C/G^0)$. This shows that these surfaces belong to the generic covers considered in \cite{PolMono}*{Theorem 2}. Moreover the three families with $K^2=6$  give some evidences to the conjecture in \cite{PolMono}*{2.4} that all generic covers of degree $4$ branched on that curve should have $p_g=q=2$. We hope to come back to this question in a future paper.

\subsubsection*{Notation and conventions.} We work over the field
$\mathbb{C}$ of complex numbers. 
By \emph{curve} resp. \emph{surface} we mean a smooth projective variety of dimension $1$ resp.  $2$.

For every curve $C$ we write $J(C)$ for the principally polarized Abelian variety known as the Jacobian of $C$, $\Pic^d(C)$ for the family of the line bundles of degree $d$ on $C$,  $\Sym^d(C)$ for the $d$-th symmetric product of $C$. We will write the class of $(p_1,\ldots,p_d)$ in $\Sym^d(C)$ as $p_1+\ldots +p_d$.
Following \cite{BL94}*{Chapter 11, Section 1} there is a canonical map $\rho_d(C) \colon \Sym^d(C) \rightarrow  \Pic^d(C)$ associating to  $p_1+\ldots + p_d$ the line bundle ${\mathcal O}_C(p_1+\ldots +p_d)$. We will usually identify $\Pic^d(C)$ with $J(C)$ by means of the isomorphism given by the Abel-Jacobi map (although this map depends on the choice of a line bundle on $C$ of  degree $d$, a choice that is not relevant for our computations).  

For every surface $S$, $K_S$ denotes a canonical divisor of $S$, $p_g(S)=h^0({\mathcal O}_S(K_S))$ is  its \emph{geometric genus},
and $q(S)=h^1( {\mathcal O}_S)$ is its \emph{irregularity}.
Given a finite group $G$ acting on a vector space $V$, we denote by $V^{G}$ the  subspace of $G$-invariant elements.

A dominant generically finite morphism $f \colon A \rightarrow B$ among smooth varieties is {\it simply ramified} if at the general point of each irreducible component of the ramification locus  it is analytically of the form $(x_1,\ldots,x_{n-1},x_n) \mapsto (x_1,\ldots,x_{n-1},x_n^2)$. 
A dominant generically finite morphism $f \colon A \rightarrow B$ among smooth varieties is {\it simply branched} if it is simply ramified and the preimage of the general point of each irreducible component of the branching locus contains only a ramification point.

A double (resp. triple) cover is a dominant generically finite morphism of degree  $2$ (resp. $3$).

For every group $G$ and for very element $g\in G$ we denote by $\Ad_g$ the internal automorphism of $G$ defined by $g$ via conjugation: $\Ad_g(h)=ghg^{-1}$.

For every group $H$, we denote by $H(2)$ the semidirect product $(H \times H) \rtimes \ZZ_{/2\ZZ}$ where the conjugacy action of a generator $\sigma$ of $\ZZ_{/2\ZZ}$ on $H \times H$ is $\Ad_{\sigma}(h_1,h_2)= (h_2,h_1)$.

For every  abelian group $H$, the generalized dihedral group $D(H)$ is the  semidirect product  $H \rtimes \ZZ_{/2\ZZ}$ similarly given by $\Ad_\sigma(h)=h^{-1}$.


\numberwithin{theorem}{section}
\section{Mixed surfaces}\label{MixedSurfaces}
In this section we introduce the mixed surfaces and few general results on them.

\begin{definition}\label{mixedaction}{\em (see \cite{Fr13}*{Definition 3.1} and \cite{Ca00}*{Proposition 3.15})}
Let $C$ be a curve, and let $\sigma$ be the involution on $C \times C$ exchanging the factors: for all $x,y \in C$, $\sigma (x,y)=(y,x)$. Set $\Aut(C)^2 \subset \Aut (C \times C) $ for the subgroup of automorphisms of $C \times C$ preserving the factors, and 
$\Aut(C)(2) \subset \Aut (C \times C)$ for the smallest subgroup containing $\Aut(C)^2 $ and $\sigma$.

A {\em mixed action} on $C \times C$ is a finite subgroup $G$ of $\Aut(C)(2)$ not contained in $\Aut(C)^2$.

We denote by $G^0$ its index two subgroup $G \cap  \Aut(C)^2$.
A mixed action is {\em minimal} if $G^0$ acts faithfully on both factors.
\end{definition}

To be precise, \cite{Fr13} and \cite{Ca00} work under the assumption that $C$ has genus at least $2$, since that implies (\cite{Ca00}*{Corollary 3.9}) $\Aut(C \times C) \cong \Aut(C)(2)$. Anyway, their results we quote do not need that assumption, so we may write Definition \ref{mixedaction} in this slightly more general setting.

\begin{definition}{\em (see \cite{Fr13}*{Definition 3.2})}
A mixed quotient is a surface which arises as a quotient $X:= \left( C \times C \right) /G$ by a mixed action of $G$ on
$C \times C$.
\end{definition}

As shown in \cite{Fr13}*{Remark 3.3} (cf. \cite{Ca00}*{Remark 3.10, Proposition 3.13}), every mixed quotient
$X$ is obtained by an unique minimal mixed action. So we may assume without loss of generality the action to be minimal.

\begin{definition}
A mixed surface $S$ is the surface obtained by taking the minimal resolution $r \colon S \rightarrow X$ of the singularities of a mixed quotient $X$.
\end{definition}

We will use the following explicit description of all minimal mixed actions given for the first time in \cite{Ca00}*{Proposition 3.16}
(see also \cite{Fr13}*{Theorem 3.6}).

\begin{theorem}\label{charmixquot}(\cite{Ca00}*{Proposition 3.16})
Let $G \subset \Aut(C \times C) $ be a minimal mixed action. Choose $\tau' \in G \setminus G^0$ and set $\tau=(\tau')^2 \in G^0$. The group $G^0$,
being normal in $G$, is  preserved by the conjugation by $\tau'$, defining $\varphi \in \Aut (G^0)$ via $\varphi(h)= \Ad_{\tau'}( h)$.

Then, up to a coordinate change, there is a faithful action $G^0 \hookrightarrow \Aut (C)$ such that $G$ acts as follows
\begin{equation}\label{explicit mixed action}
\begin{cases}
g(x,y)=(gx,\varphi(g)y)\\
\tau' g(x,y)=(\varphi(g)y, \tau gx)
\end{cases}
\end{equation}

Conversely, for every $G^0 \subset \Aut( C)$ and for every  extension $G$ of degree $2$ of $G^0$, once chosen $\tau' \in G \setminus G^0$ and accordingly defined $\tau$ and $\varphi$ as above, (\ref{explicit mixed action}) defines a minimal mixed action of $G$ on $C\times C$.
\end{theorem}

If the exact sequence 
\[
1 \rightarrow G^0 \rightarrow G \rightarrow \ZZ/2\ZZ \rightarrow 1,
\]
splits one can assume $\tau=1$ by choosing $\tau'$ of order $2$. We recall that the sequence does not split if and only if the quotient map $C \times C \rightarrow X$ is quasi-\'etale (\cite{Fr13}*{Theorem 3.7}).


A minimal mixed action $G \subset \Aut (C \times C)$ is naturally a subgroup of a bigger group  of automorphisms ${G^0(2)} \subset \Aut(C \times C)$ as follows.

\begin{definition}\label{Gstorto}
Let  $G \subset \Aut (C \times C)$ be a minimal mixed action, and let $G^0$ be the index $2$ subgroup of $G$ given by the elements acting separately on the factors. 

Consider the natural inclusion of  $G^0 \times G^0$ in $\Aut (C \times C)$ defined by $(g,h)(x,y)=(gx,hy)$. 
We denote by ${G^0(2)} \cong \left( G^0 \times G^0 \right) \rtimes \ZZ/2\ZZ$ the smallest subgroup of $\Aut(C \times C)$ containing $G^0 \times G^0$ and $\sigma$.
\end{definition}

\begin{remark}\label{dominating a Sym}\

\begin{enumerate}
\item The group $G$ is a subgroup of ${G^0(2)}$ of index $|G^0|$, the inclusion being given by, for all $g \in G^0$
\begin{align*}
g \mapsto& (g, \varphi(g)), & \tau' g \mapsto& \sigma \circ  (\tau g, \varphi(g)).
\end{align*}
\item There is a natural isomorphism 
\[
\frac{C \times C }{G^0(2)} \cong \Sym^2(C/G^0).
\]
\end{enumerate}
\end{remark}

Since the quotient by a group factors through the quotient by a subgroup, it follows that there is a finite dominating morphism of degree $|G^0|$ 
\begin{equation}\label{piG0}
\pi_{G^0} 	\colon X \rightarrow \Sym^2(C/G^0).
\end{equation}

In the rest of this section we discuss some applications of this map.

The following is a special case of \cite{ENR}*{Proposition 1.6} (see also  \cite{Ca00}*{Proposition 3.15}, \cite{MisPol}*{Proposition 3.5}, \cite{penego}*{Proposition 3.14}, \cite{Fr13}*{Lemma 3.9}) and \cite{CanFr15}*{Proposition 3.6}).
\begin{proposition}\label{g=q}
Let $S$ be a mixed surface, minimal resolution of a mixed quotient $r \colon S \rightarrow X=(C \times C)/G$ where $G \subset \Aut(C \times C)$  is a minimal action. 

Consider $G^0$ as subgroup of $\Aut(C)$ through its action on the first factor. Then
\[
(\pi_{G^0} \circ r)^* \colon H^0(\Omega^1_{\Sym^2(C/G^0)}) \rightarrow H^0(\Omega^1_S)
\]
is an isomorphism.

In particular the irregularity of $S$ equals the genus of the quotient curve:
\[
q(S)=g(C/G^0).
\]
\end{proposition}

Let $C'$ be a curve of genus $q \geq 2$. It is well known that the Abel-Jacobi map $\rho_2(C')\colon \Sym^2 (C') \rightarrow J(C')$ is generically injective. More precisely, it is an embedding if and only if $C'$ is not hyperelliptic. If $C'$ is hyperelliptic, it is the contraction of the  rational curve given by the divisors in the unique $g^1_2$, a rational curve with self-intersection $1-q$.

Then, if $C/G^0$ is not hyperelliptic, $S$ does not contain any rational curve. 

\begin{proposition}\label{norat}
Let $X=(C \times C)/G$ be a mixed quotient such that $C/G^0$ is not hyperelliptic. Then $X$ contains no rational curves.
\end{proposition}
\begin{proof}
Assume that $E$ is a rational curve in $X$. Since $\pi_{G^0}$ is a finite morphism, then $\pi_{G^0}(E)$ is a rational curve in $\Sym^2(C/G^0)$.
Since $C/G^0$ is not hyperelliptic then $\rho_2(C/G^0)$ embeds $\Sym^2(C/G^0)$ in $J(C/G^0)$ that does not contain any rational curve, a contradiction. 
\end{proof}


\begin{theorem}\label{minimality}
Let $S$ be a mixed surface with irregularity $q(S) \geq 3$. Then $S$ is minimal.
\end{theorem}

\begin{proof} If $C'$ is not hyperelliptic by Proposition \ref{norat} $X$ does not contain any rational curve, and then all rational curves of $S$ are contracted by $r \colon S \rightarrow X$. Since $r$ is the minimal resolution of the singularities then $S$ does not contain $(-1)$-curves and therefore $S$ is minimal. 

So we may assume $C'$ hyperelliptic. Since by Proposition \ref{g=q} the curve $C'$ has genus  $q(S) \geq 3$, a general small deformation of $C'$ is not hyperelliptic. 
Choose then a family ${\mathcal C}' \rightarrow \Delta$ on a $1-$ dimensional small disc whose central fibre is $C'$ and whose general fibre is not hyperelliptic. 

Theorem \ref{charmixquot} shows that mixed surfaces are obtained by choosing
\begin{enumerate}[(i)]
\item a curve $C'$ (that will be $C/G^0$);
\item a Galois cover $C \rightarrow C'$ (of Galois group  $G^0$);
\item a degree $2$ extension $G$ of $G^0$.
\end{enumerate}

These give an action of $G$ on $C \times C$, and the mixed surface is obtained as minimal resolution of the singularities of $X = (C \times C)/G$.

The datum (ii) is determined by purely topological data on $C'$, a finite subset $\{p_j\}$ (the branch locus) and a subgroup of the fundamental group of its complement. 
Choosing sections of ${\mathcal C}' \rightarrow \Delta$ passing through each $p_j$ one obtains a $G^0$-cover ${\mathcal C} \rightarrow {\mathcal C'}$  over $\Delta$ branched on them whose central fibre is the original cover  $C \rightarrow C'$.

The datum iii) defines then as in (\ref{explicit mixed action})  a $G$-action on the fibre product ${\mathcal C} \times_\Delta {\mathcal C}$ over $\Delta$. The quotient is a family of mixed quotients ${\mathcal X} \rightarrow \Delta$   with central fibre $X$. 

Let $p$ be a point of ${\mathcal C} \times_\Delta {\mathcal C}$ . By construction there are local coordinates $(x,y,z)$  near $p$ such that $z$ gives the map on $\Delta$  and the action of each element of $G$ stabilizing $p$ is of the form $(x,y,z)\mapsto(f(x,y),g(x,y),z)$. This implies that  ${\mathcal X}$ is analytically equisingular: its singular locus is a disjoint union of sections of ${\mathcal X} \rightarrow \Delta$ where the singularity of   ${\mathcal X}$   is the product of an isolated surface singularity by $\Delta$. Then the minimal resolution of the isolated surface singularities gives a resolution ${\mathcal S} $ of the singularities of $ {\mathcal X}$ giving on each point of $\Delta$ the minimal resolution of the singularities of the corresponding fibre of ${\mathcal X}\rightarrow \Delta$.

Then ${\mathcal S} \rightarrow \Delta$ is a flat family  deforming $S$ to a mixed surface arising by a deformation of $C'$ that is not hyperelliptic. In particular by the argument at the beginning of this proof we have found a small deformation of $S$ that is minimal. By Kodaira's theorem on stable submanifolds (\cite{Kod63}*{Example at page 86}) $S$ is minimal too.
\end{proof}

\section{Subgroups of \texorpdfstring{$G^0(2)$}{G0(2)} containing \texorpdfstring{$G$}{G}}\label{GroupTheory}
 
In this section we study the subgroups $G'$ of ${G^0(2)}$ containing $G$ from a purely algebraic point of view.

Consider an extension of finite groups 
\[
1 \rightarrow G^0 \rightarrow G \rightarrow \ZZ/2\ZZ \rightarrow 1,
\]
and fix once for all an element $\tau' \in G \setminus G^0$. Set $\tau:=(\tau')^2 \in G^0$.

We deduce from $\tau'$ the injective homomorphism $G \hookrightarrow G^0(2)$ such that
\begin{align*}
\tau' \mapsto& \sigma (\tau,1)=(1,\tau)\sigma,&
g \mapsto& (g, \Ad_{\tau'}(g)) \text{ for all } g \in G^0.
\end{align*}
where $\sigma \in G^0(2) \setminus (G^0 \times G^0)$ is an element such that. for all $g_1,g_2 \in G^0$, $\Ad_\sigma ( (g_1,g_2) )=(g_2,g_1)$. 

We  identify $G$ with its image, so considering $G$ as a subgroup of $G^0(2)$. We show now that the subgroups of $G^0(2)$ containing $G$ are determined by their subgroup of the elements acting trivially on the second factor. 

\begin{theorem}\label{G_K}
For every subgroup $G'$ of $G^0(2)$ containing $G$ consider the intersection $G' \cap \left (G^0 \times \{1\}\right)=:K \times \{1\}$ of the elements acting trivially on the second factor.

Then $K$  is normal in $G$. 

Moreover $K$ determines $G'$ in the sense that for every subgroup $K$ of $G^0$ normal in $G$ there is a unique group $G_K$, $G \subset G_K \subset G^0(2)$, such that $G_K \cap \left (G^0 \times \{1\}\right) = K \times \{1\}$.
\end{theorem}
\begin{proof}
We first show that $K$ is normal in $G^0$. Indeed, for all $g \in G^0$ and for all $k \in K$, 
 \[
(gkg^{-1},1)= (g,\Ad_{\tau'}(g)) (k,1) (g,\Ad_{\tau'}(g))^{-1}.
 \]
Since both $(g,\Ad_{\tau'}(g))$ and $(k,1)$ belong to $G'$ then $(gkg^{-1},1)$ belongs to $G'$ too. So $gkg^{-1}$ belongs to $K$ 
and therefore  $K$ is normal in $G^0$.

We similarly show that $(\Ad_{\tau'}(k),1)$ belongs to $G'$ by writing it as product of elements of $G'$ as follows: 
\[
(\Ad_{\tau'}(k),1) = \sigma (\tau,1)(1,\Ad_{\tau'}(k))(\tau^{-1},1) \sigma=\tau' (k,\Ad_{\tau'}(k)) (k^{-1} ,1)(\tau')^{-1}.
\]
Then $\Ad_{\tau'}(k)$ belongs to $K$ for all $k$ in $K$. This shows that  $K$ is normal also in $G$.

\bigskip

Now we prove that $K$ determines $G'$. 

We construct, for every subgroup $K$ of $G^0$ normal in $G$, the subgroup  $G_K$ of $G^0(2)$. 
We first define
\[
G^0_K:=\left\{ (g,h) \in G^0 \times G^0 |\ h\Ad_{\tau'}(g)^{-1} \in K
\right\}.
\]
We prove that $G^0_K$ is a subgroup of $G^0 \times G^0$ by showing that for each pair of elements $a_1,a_2 \in G^0_K$, the product $a_1a_2^{-1}$ belongs to $G_K^0$ as well. Indeed, writing $a_j=(g_j,h_j)$, since $a_1a_2^{-1}=(g_1g_2^{-1},h_1h_2^{-1})$ we need to show that  $h_1h_2^{-1}\Ad_{\tau'}(g_1g_2^{-1})^{-1}$ belongs to $K$. This follows by writing it
as
\[
 h_1h_2^{-1}\Ad_{\tau'}(g_1g_2^{-1})^{-1}=\left[ (h_1h_2^{-1}) ( h_2\Ad_{\tau'}(g_2)^{-1})^{-1}  (h_1h_2^{-1})^{-1}  \right] h_1\Ad_{\tau'}(g_1)^{-1}
\] using that both $a_j$ belong to $G^0_K$ and the normality of $K$ in $G^0$.

Now we define $G_K$ as subset of $G^0(2)$ by
\begin{equation}\label{defG_K}
G_K:=G^0_K \cup  \left( \tau' G^0_K \right).
\end{equation}
We notice that all elements of the form $(g, \Ad_{\tau'}(g))$  belong to $G^0_K$. In particular
$(\tau')^2$ belongs to $G^0_K$. Then we may prove that $G_K$ is a subgroup of $G^0(2)$ by showing that $\tau' G^0_K=G^0_K \tau'$.

We do that. By  $\tau'(g,h)=(h,\tau g \tau^{-1}) \tau'$ we have to prove that if $(g,h)$ belongs to $G_K^0$ then $(h,\tau g \tau^{-1})$ belongs to $G_K^0$ too.
This follows immediately by the equality $\tau g \tau^{-1} \Ad_{\tau'}(h)^{-1}=
\Ad_{\tau'} \left(  h \Ad_{\tau'}(g^{-1})\right)^{-1}$ using the  normality of $K$ in $G$.

We have then proved that $G_K$ is a subgroup of $G^0(2)$. Since $G_K$ contains $\tau'$ and all elements of the form $(g, \Ad_{\tau'}(g))$, it contains $G$. Moreover $G_K \cap \left( G^0 \times \{1\}\right)=G^0_K \cap \left( G^0 \times \{1\}\right)$ equals $K \times \{1\}$: indeed $(g,1) \in G^0_K$ if and only if $\Ad_{\tau'}(g)^{-1}\in K$ that is equivalent by the normality of $K$ to $g \in K$.
This completes the proof of  the existence of $G_K$.

\smallskip

Finally we prove the uniqueness of $G_K$. 

We notice that  $G \cap (G^0 \times \{1\} )$  is trivial. This implies that two distinct elements of $G^0 \times \{1\}$ belong to disjoint left cosets $hG$.
Then the natural map from $G^0 \times \{1\}$ to the set of left cosets $\{hG\}$ is, 
by  $[{G^0(2)}:G]=|G^0|$, a bijection.
Since every subgroup of $G^0(2)$ containing $G$ is union of left cosets $hG$ then at most one of them has the property  $G_K \cap \left( G^0 \times \{1\}\right)=K \times \{1\}$, as subset of $G^0(2)$ it is the union of the cosets $(k,1)G$ varying $k$ in $K$.
\end{proof}

\begin{remark}\label{KxK}
By the proof  of Theorem \ref{G_K} it follows that  
$K \times K $ is a normal subgroup of $G_K$ and $\frac{G_K}{K \times K} \cong G/K$.

Indeed $K \times K \subset G^0_K$ by definition of $G^0_K$. Since $K$ is normal in $G^0$, $K \times K$ is normal in $G^0 \times G^0$ and therefore, since it is invariant by the exchange of the factors, also in ${G^0(2)}$ and then  in the smaller subgroup $G_K$.

In the final part of the proof of Theorem \ref{G_K} we have shown that $G_K$ is the union of $|K|$ left cosets of $G$, so $|G_K|=|G|\cdot |K|$. Then $\left| \frac{G_K}{K \times K} \right| = \frac{|G_K|}{|K|^2}=\frac{|G|\cdot|K|}{|K|^2}=\frac{|G|}{|K|}=|G/K|$.

The  inclusion $G\subset G_K$ induces a homomorphism from $G$ to $\frac{G_K}{K \times K}$ whose  kernel is isomorphic to $K$. 
Then it is also surjective and it induces an isomorphism $ \frac{G_K}{K \times K}  \cong G/K$.
\end{remark}

We now determine for which $K$ the group $G_K$ is normal in $G^0(2)$.

\begin{proposition}\label{normality}
The group $G_K$ is normal in $G^0(2)$ if and only if $K$ contains the commutator subgroup $[G^0,G^0]$
and all elements of the form $g\Ad_{\tau'}(g)$, for $g \in G^0$.
\end{proposition}
In other words $G_K$ is normal in ${G^0(2)}$  if and only if  $G^0/K$ is abelian and $G/K$ is the corresponding generalized dihedral group $D(G^0/K)$. 

Indeed, $g\Ad_{\tau'}(g) \in K$ means that the conjugation by the class of $\tau'$ on $G^0/K$ maps $g$ to its inverse.
\begin{proof}
$G_K$ is generated by $\tau'$, the elements  $(h,\Ad_{\tau'}(h))$ for $h \in G^0$ and the elements $(k,1)$ for $k \in K$. 

Since ${G^0(2)}$ is generated by $G_K$ and the elements of the form $(g,1)$, $g \in G^0$, $G_K$ is normal in ${G^0(2)}$ if and only the conjugates of all elements $\tau'$, $(h,\Ad_{\tau'}(h))$, $(k,1)$ above by any $(g,1)$ belong to $G_K$.

Since $(g,1) \tau' (g,1)^{-1}=\tau'( g^{-1},g)$, then by (\ref{defG_K}) the conjugates of $\tau'$  by any $(g,1)$ belong to $G_K$ if and only if $\{ g\Ad_{\tau'}(g) | g \in G^0 \} \subset K$.
 
Note that $(g,1)(h,\Ad_{\tau'}(h))(g,1)^{-1}=(ghg^{-1},\Ad_{\tau'}(h))$ belongs to $G_K$ if and only if $\Ad_{\tau'}(hgh^{-1}g^{-1}) \in K$.
Since $\Ad_{\tau'}(K)=K$, the conjugates of $(h,\Ad_{\tau'}(h))$  by any $(g,1)$ belong to $G_K$ if and only if  $[G^0,G^0]\subset K$. 

Finally the conjugates of elements of type $(k,1)$  by any $(g,1)$ are always in $G_K$ by the normality of $K$ in $G^0$.
\end{proof}

\section{Further quotients}\label{FurtherQuotients}

Let $C$ be a Riemann surface, let $G \subset \Aut(C \times C)$ be  a minimal mixed action, and consider $G^0$ as subgroup of $\Aut(C)$ by means of the action on the first factor.

By Theorem \ref{charmixquot}, up to a coordinate change, $G$ is contained in the subgroup of $\Aut(C \times C)$ isomorphic to $G^0(2)$ generated by $G^0 \times G^0$ and the involution $\sigma$ exchanging the factors.  More precisely it is the subgroup generated by the element of the form $(g,\Ad_{\tau'}(g))$ (for $g \in G^0$) and $\sigma(\tau,1)$.

So the inclusion $G \hookrightarrow G^0(2)$ is as in Section \ref{GroupTheory}.

Theorem \ref{G_K} produces then groups of automorphisms $G_K$ of $C \times C$, $G \subset G_K \subset G^0(2)$, each giving a  factorization 
of the map $\pi_{G^0}$ in (\ref{piG0}).

\begin{definition}\label{pi_K}
For each  subgroup $K \subset G^0$ normal in $G$ we define the {\em further quotient by $K$} as the quotient of $C \times C$ by the group $G_K$.

The map $\pi_{G^0}$ factors through the further quotient by $K$ as  follows
\[
\frac{C \times C}{G} 
\stackrel{\pi_K}{\longrightarrow} \frac{C \times C}{G_K} \stackrel{\rho_K}{\longrightarrow}  \frac{C \times C}{G^0(2)}\cong  \Sym^2 (C/G^0).
\]
\end{definition}
 Note $\deg \pi_K=|K|$. Note also that $\pi_K$ and $\rho_K$ are not necessarily Galois covers, since $G$ may be not normal in $G_K$ and $G_K$ may not be normal in ${G^0(2)}$.

We  compute the minimal realization of  $\frac{C \times C}{G_K}$ (compare \cite{Fr13}*{Remark 2.3}) using Remark \ref{KxK}.
Indeed, since $K \times K$ is a normal subgroup of $G_K$, then the quotient map $C \times C \rightarrow (C \times C)/G_K$ factors through $(C \times C)/(K \times K)\cong C/K \times C/K$, and the residual map $C/K \times C/K \rightarrow (C \times C)/G_K$ is the quotient by the resulting action  of $\frac{G_K}{K \times K} \cong G/K$ on $C/K \times C/K$.

It is now easy to see that this action is a minimal mixed action, giving the minimal realization of the further quotient as $\frac{C/K \times C/K}{G/K}$.

\section{Albanese morphisms of mixed surfaces}

In this section we give a complete description of the Albanese morphism of a mixed surface by factoring it through a suitable further quotient.

We first need a couple of rather standard lemmas. We give here a proof of them by lack of a complete reference.

\begin{lemma}\label{liftingthenquotient}
Let $g \colon X \rightarrow Z$ be a Galois covering with Galois group $G$, and let $h \colon Y \rightarrow Z$ be  a Galois covering with Galois group $H$.

Assume that there is a morphism $f \colon X \rightarrow Y$ such that $g = h \circ f$. Then $f$ is the quotient by a normal subgroup $K$ of $G$ and $H \cong G/K$. 
\end{lemma}
\begin{proof}
Choose a smooth point $p_Z$ in $Z$ that is a regular value for both $g$ and $h$. Choose $p_X \in g^{-1}(p_Z)$. 
We define a map $\Theta \colon G \rightarrow H$ associating to  every $\phi \in G$ the unique $\Theta(\phi):=\psi \in H$ determined by the property $f \circ \phi (p_X)= \psi \circ f(p_X)$ . Then (cf.  \cite{Massey}*{Lemma V.3.2})  $f \circ \phi = \psi \circ f$ since they are two liftings of the same map coinciding in a point.

Then set $K:= \ker \Theta$. For all $\phi \in K$ it holds $f \circ \phi =  f$. So $K \subset \Aut(f)$ and then $f$ factors through the quotient by $K$. The resulting map $X/K \rightarrow Y$ is finite and of degree $1$ and therefore it is an isomorphism.
\end{proof}

\begin{lemma}\label{ro2inducesiso}
Let $C$ be a curve. Then  
\[
\rho_2(C)_* \colon \pi_1 (\Sym^2C) \rightarrow \pi_1(J(C))\cong \ZZ^{2g(C)}
\] is an isomorphism.
\end{lemma}
\begin{proof}
By \cite{Topologyandgroupoids}*{11.5.4} or Armstrong Theorem \cite{Armstrong} the map $C \times C \rightarrow \Sym^2 C$ 
induces a surjection on fundamental groups with kernel normally generated by the loops of the form $(\gamma,\gamma^{-1})$: in other words $\pi_1(Sym^2(C))$ is the quotient of $\pi_1(C)^2$ by the relation $(\gamma,1) \sim (1,\gamma)$.

 In particular the fundamental group of $ \Sym^2C$ is abelian, isomorphic to the abelianization of the fundamental group of $C$, $H_1(C,\ZZ)\cong \ZZ^{2g(C)}$. More precisely, there is an  isomorphism mapping the homotopy class of the image in  $\Sym^2(C)$ of a loop $\gamma \times \{p\}$ in $C \times C$ to the homology class of $\gamma$ in $H_1(C,\ZZ)$. 

Since $\rho_1(C) = \rho_2(C) \circ f_p$ for $f_p \colon C \rightarrow \Sym^2(C)$ defined via $f_p(q)=q+p$  and it is 
well known that $\rho_1(C)$, the natural map from a curve to its Jacobian, induces an isomorphism among the first homology groups with integral coefficients, the result follows.
\end{proof}

We are now able to prove our main theorem.

\begin{theorem}\label{Albanese and generalized dihedral surfaces}
Let $S$ be an irregular mixed surface, minimal resolution of a mixed quotient $r \colon S \rightarrow X= (C \times C)/G$ where $G \subset \Aut(C \times C)$  is a minimal action. 

We consider the index two subgroup $G^0$ as a subgroup of $\Aut(C)$ by its action on the first factor. Let $\Sigma \subset G^0$ be the set of the elements fixing at least one point of $C$.

Let $K \subset G^0$ be the smallest  subgroup of $G^0$ normal  in $G$ containing $\Sigma$, the commutator subgroup $[G^0,G^0]$ and $\{ g^2 | g \in G \setminus G^0\}$.

Then $\Alb(S)$ is isogenous to $ J(C/G^0)$ via an  isogeny with kernel isomorphic to $G^0/K$.  Moreover,  if $q(S) \geq 2$, then $S$ is of maximal Albanese dimension and the degree of its Albanese morphism equals $|K|$.

More precisely, there is a commutative diagram
\begin{equation*}   
   \xymatrix{ 
S\ar[d]^{\alpha_S}\ar[r]^{r}&        \frac{C \times C}{G} \ar[r]^{\pi_{K}}\ar[d]^{\alpha_{ \frac{C \times C}{G}}}&    \frac{C \times C}{G_K} \ar[r]^{\rho_K}\ar[d]^{\alpha_{ \frac{C \times C}{G_K}}}&  \Sym^2 (C/G^0)\ar[d]^{\rho_2(C/G^0)} \\  
\Alb(S)\ar@{=}[r] &     \Alb \left(  \frac{C \times C}{G} \right) \ar@{=}[r]&\Alb\left(  \frac{C \times C}{G_K} \right) \ar[r]^{\Alb(\rho_{K})}& J(C/G^0).}
\end{equation*}
whose vertical maps are Albanese morphisms and
 $\alpha_{ \frac{C \times C}{G_K}}$  is generically injective when $q(S) \geq 2$, whereas if $q=1$ all its fibres are rational. Moreover $\rho_K$ and $\Alb(\rho_K)$ are Galois \'etale covers with Galois group $G^0/K$.
\end{theorem}

\begin{proof}
Since $\rho_2(C/G^0) \circ \pi_{G^0} \circ r$ is a morphism from $S$ to the Abelian variety $J(C/G^0)$, it factors through the Albanese morphism of $S$. Since by Proposition \ref{g=q} $(\pi_{G^0} \circ r)^*$ is an isomorphism, the induced map $\Alb(S) \rightarrow J(C/G^0)$ is an isogeny.

By Lemma \ref{ro2inducesiso} every \'etale cover of  $J(C/G^0)$ induces a unique  \'etale cover of 
$\Sym^2(C/G^0)$ such that the map $\rho_2(C/G^0)$ lifts  to a map among the covers.
This gives a Galois \'etale cover $Z \rightarrow \Sym^2(C/G^0)$  with a map $Z \rightarrow \Alb(S)$ that lifts the map $\Sym^2(C/G^0) \rightarrow J(C/G^0)$ and therefore inherits some of its properties: it is generically injective if $q(S) \geq 2$ and with all fibres rational if $q(S)=1$.

By Lemma \ref{liftingthenquotient} there is a normal subgroup $G'$ of ${G^0(2)}$ containing $G$ such that  $Z=(C \times C) /G'$. 
By Theorem \ref{G_K} and  Proposition \ref{normality}  $G'=G_{K'}$ for a subgroup $K'$ of $G^0$ normal in $G$ containing $[G^0,G^0]$ and all elements of the form $h\Ad_{\tau'}(h)$.

Moreover $\rho_{K'} \colon Z \rightarrow \Sym^2(C/G^0)$ is  \'etale. Then all elements of ${G^0(2)}$ fixing a point of $C \times C$  must be contained in $G_{K'}$. So $\Sigma \times \{1\}\subset G_{K'}$ and then $\Sigma \subset K'$. Moreover $\sigma \in G_{K'}$, so $\sigma\tau' =(\tau,1) \in G_{K'}$, so $\tau \in K'$.
Finally note that $h\Ad_{\tau'}(h)\tau=(h\tau')^2$, so  $\{ g^2 | g \in G \setminus G^0\}\subset K'$.

So $K \subset K'$. Note that the kernel of the isogeny $\Alb(S) \rightarrow J(C/G^0)$ is isomorphic to $G^0/K'$.

We note that $G_{K}$ contains all elements in ${G^0(2)}$ fixing some points in $C \times C$. Indeed, for all  $ g,h\in G^0$, if $(g,h)$ fixes a point of $C \times C$ then $g \in \Sigma \subset K$, $h \in \Ad_{\tau'}^{-1}(\Sigma) \subset K$, so $(g,h) \in K \times K \subset G_{K}$. Similarly if $\tau'(g,h)=\sigma(\tau g,h)$ fixes a point $(x,y)$, $y=\tau ghy$ so $\tau gh \in \Sigma \subset K$, then $gh \in K$ and therefore $
\Ad_{\tau'}^{-1}(h)g^{-1}= \Ad_{\tau'}^{-1} \left( h \Ad_{\tau'}(h) \right) (gh)^{-1} \in K$: this
proves that  $\tau'(g,h)$ belongs to $G_{K}$.

This implies that $\rho_{K}$ is \'etale and then arguing as above the map $\Alb((C \times C) /G_K) \rightarrow J(C/G^0)$ is an isogeny whose kernel is isomorphic to $G^0/K$. Since the isogeny $\Alb(S) \rightarrow J(C/G^0)$ of kernel $G^0/K'$ factors through it, then $K' \subset K$ so $K=K'$.
\end{proof}

\begin{remark}\label{polarizations}
If $G^0/K \cong \ZZ/d_1\ZZ \times \ZZ/d_1d_2\ZZ \times \cdots \times \ZZ/d_1 \cdots d_q \ZZ \
$, the pull back of the Hodge form  giving a principal polarization on $J(C/G^0)$ divided by $d_1$ provides $\Alb (S)$ with a polarization of type $(1,d_2,\ldots,d_q)$.
\end{remark}

We discuss now some properties of the very special mixed surface $(C \times C )/G_K$  in Theorem \ref{Albanese and generalized dihedral surfaces}. 
Let us first recall \cite{CanFr15}*{Definition 3.1}.
\begin{definition}\label{semisogenous}
A semi-isogenous mixed surface is a mixed quotient $X:= \left( C \times C \right) /G$ such that $G^0$ acts freely on $C \times C$.
\end{definition}
The following is proved in \cite{CanFr15}*{Corollary 2.11}.
\begin{proposition}
Let $X$ be a semi-isogenous mixed surface. Then $X$ is smooth.
\end{proposition}

Theorem \ref{Albanese and generalized dihedral surfaces} suggests to consider the following smaller class of surfaces.

\begin{definition}\label{GDS}
Let $G^0$ be an abelian group acting freely on a curve $C$.

The dihedral surface $X_{C,G^0}$ is the semi-isogenous mixed surface quotient of $C \times C$ by the induced mixed action of the  generalized dihedral group $D(G^0)$.
\end{definition}

Indeed:
\begin{proposition}
The surface $(C \times C)/G_K$ in Theorem \ref{Albanese and generalized dihedral surfaces} is a dihedral surface $X_{C/K,G^0/K}$.
\end{proposition}

\begin{proof}
The minimal realization of $(C \times C)/G_K$  is $(C/K \times C/K)/(G/K)$. 

Since $K$ contains $[G^0,G^0]$ and $\Sigma$, $G^0/K$ is abelian and acts freely on $C/K$.

Finally $K$ contains all elements of the form $h\Ad_{\tau'}(h)$ with $h \in G^0$, since they equal a product of squares $(h\tau')^2(\tau')^{-2}$ of elements in $G \setminus G^0$. Then the conjugation by the class of $\tau'$ acts on $G^0/K$ as $h \mapsto h^{-1}$.  Then $G/K \cong D\left(G^0/K \right)$.
\end{proof}

\begin{remark}\label{invZ2}
We compute the main invariants of the dihedral surfaces using the formulas in \cite{CanFr15}*{Section 3} (compare also Proposition \ref{g=q}) obtaining 
\begin{align*}
q(X_{C,G^0})=&\frac{g(C)-1}{G^0}+1,& \chi ({\mathcal O}_{X_{C,G^0}})=& \frac{|G^0|(q-1)(q-2)}{2},& K^2_{X_{C,G^0}}=&|G^0|(q-1)(4q-9).
\end{align*}
\end{remark}

\begin{example}\label{FabFrank}
Consider the dihedral surfaces $X_{C,G^0}$ with $|G^0|=2$, $g(C)=5$. Then by Theorem \ref{minimality}, Remark \ref{invZ2} and Theorem \ref{Albanese and generalized dihedral surfaces} ($K=\{1\}$), they are minimal surfaces of general type with $K_{X_{C,G^0}}^2=12$, $p_g=4$ and $q=3$ and Albanese morphism generically injective in an Abelian threefold isogenous to $J(C/G^0)$ via an isogeny of degree $2$. So by Remark \ref{polarizations} their Albanese varieties have a polarization of type $(1,1,2)$.

The surfaces with these invariants and Albanese morphism generically injective on an Abelian threefold with polarization $(1,1,2)$ have been studied in  \cite{CataneseSchreyer}, proving in particular that they form a subvariety of dimension $7$ of the Gieseker moduli space of the surfaces of general type. 
Our examples form a $6$-dimensional family,  since the map $X_{C,G^0} \mapsto C/G^0$ is finite and dominating the moduli space of curves of genus $3$. In fact
 \cite{CataneseSchreyer} noticed {\em a remarkable divisor} in their family, formed by surfaces whose Albanese variety is a double cover of a Jacobian. Then the general surface in this  divisor is a dihedral surface. 
 
 This shows that a small deformation of a dihedral surface is not necessarily a dihedral surface.
\end{example}

\section{Ramification of further quotients}
In the next and last section we will apply Theorem \ref{Albanese and generalized dihedral surfaces} to describe the Albanese morphism of the semi-isogenous surfaces with $p_g=q=2$ constructed in \cite{CanFr15}. To describe their   ramifications  we need to understand the ramification of $\pi_K$. We study it in this section for semi-isogenous mixed surfaces such that the action of $G^0$ on $C$ is free.

\begin{definition} Let $X$ be a semi-isogenous mixed surface.
For each $h \in G^0$ let $\Gamma_h=\{(x,hx))\in C \times C\}$ be the graph of $h$. For each $g \in G \setminus G^0$ define $R_g=\Gamma_{\tau'g}$.
\end{definition}
Then (\cite{CanFr15}*{Proposition 3.3})
\begin{proposition}\label{ramification}
Let $X$ be a semi-isogenous mixed surface. 

The ramification locus of the quotient map $ C \times C \rightarrow X$ 
is the disjoint union of the curves $R_g$, for $g \in O_2:=\{g \in G \setminus G^0 | g^2=1\}$. 
The stabilizer of each of these $R_g$ is $\langle g \rangle \cong \ZZ/2\ZZ$.
\end{proposition}

If the action of $G^0$ on $C$ is free, we  have a simple description of the ramification of $\pi_{G^0}$. 

\begin{proposition}\label{simply ramified}
Let $X$ be a semi-isogenous mixed surface such that the action of $G^0$ on $C$ is free. Then $\pi_{G^0}$ is simply ramified and the ramification locus is the \'etale image of the curves $\Gamma_h \subset C \times C$ such that $(\tau' h)^2 \neq 1$, {\it i.e.} of the curves $R_g$ such that $g \not\in O_2$. 
\end{proposition}
Here by {\em \'etale} image of a curve we mean the image of the curve, and that the curve does not contain any ramification point.
\begin{proof} The ramification points of the map $C \times C \rightarrow \Sym^2(C/G^0)$, quotient by the action of $G^0(2)$, are the points of $C \times C$ stabilized by a nontrivial subgroup of $G^0(2)$.

Since $G^0$ acts freely on $C$, then $G^0 \times G^0$ acts freely on $C \times C$. The elements in ${G^0(2)} \setminus (G^0 \times G^0)$ act as $(x,y)\mapsto (h'y,hx)$. Since $G^0$ acts freely, their set of fixed points is not empty only for $h'=h^{-1}$ in which case it is $\Gamma_h$. The stabilizer of $\Gamma_h$ is $\langle  (h^{-1},h)  \circ \sigma \rangle \cong \ZZ/2\ZZ$. 

So the quotient map by the action of $G^0(2)$ is simply ramified, with ramification locus $\bigcup_{h \in G^0} \Gamma_h$. Since by Proposition \ref{ramification} the quotient map $ C\times C \rightarrow X$ by the subgroup $G$ ramifies exactly on those $\Gamma_h$ such that $\tau' h \in O_2$, the {\em residual} map $\pi_{G^0}$ ramifies exactly on the image of the others.
\end{proof}

Now consider a further quotient as in Definition \ref{pi_K}, factoring $\pi_{G^0}$ as $\rho_K \circ \pi_K$.
We get  the following description of the ramification of $\pi_K$ and $\rho_K$. 

\begin{proposition}\label{ramification loci} 
Let $X$ be a semi-isogenous mixed surface such that the action of $G^0$ on $C$ is free. 
For each subgroup $K$ of $G^0$ normal in $G$, the maps $\rho_K$ and $\pi_K$ in Definition \ref{pi_K} are simply ramified.

The ramification locus of $\pi_K$ is $\bigcup_{g^2 \in K, g^2 \neq 1} R_g$.

The ramification locus of $\rho_K$ is the \'etale image of $\bigcup_{g^2 \not\in K} R_g$.
\end{proposition}

\begin{proof}
$\rho_K$ is the map $\pi_{G^0/K} \colon \frac{C/K \times C/K}{G/K} \rightarrow \Sym^2(C/G^0)$ and therefore, by Proposition \ref{simply ramified}, it ramifies on the curves $R_{[g]} \subset C/K \times C/K$, where $[g]$ is the class of $g$ in $G/K$, such that $[g]^2 \neq 1$:
these are the \'etale images of the curves  $R_g$ such that $g^2 \not\in K$.

By Proposition \ref{simply ramified} $\pi_{G^0}$ is simply ramified at the $R_g$ such that $g^2 \neq 1$. 
Since $\pi_{G^0}=\rho_K \circ \pi_K$, the ramification locus of $\pi_K$ is the union of the ramification curves of $\pi_{G^0}$ not dominating ramification curves of $\rho_K$.
\end{proof}

\section{Applications}\label{AlbaneseMorphisms}

In this final section we apply Theorem \ref{Albanese and generalized dihedral surfaces} to compute  the Albanese morphism of several mixed surfaces. We start by defining two important curves.
\begin{definition}
Let $C$ be a curve of genus $2$. Let $\eta_C \colon C \rightarrow C$ be its hyperelliptic involution. 

Then we define the curves $\Delta_C , \Gamma_C \subset \Sym^2(C)$ respectively as
$$
\Delta_C=\left\{ p+p \right\} \cong C \ \ \ \ \ \
\Gamma_C=\left\{ p+ \eta_C p \right\} \cong \PP^1.
$$
\end{definition}

Recall that $\rho_2(C) \colon \Sym^2(C) \rightarrow J(C)$ is the contraction of the $(-1)-$curve $\Gamma_C$, mapping $\Delta_C$ to a curve with an ordinary 6-ple point, image of the six Weierstra\ss\ points of $\Delta_C$.
 
 The first case is a minor reformulation of  \cite{pigpol}*{Propositions 2.5 and 2.8}.  We write it for the convenience of the reader as a warm up for the next more difficult cases.
 
\begin{proposition}\label{K^2=7}
Let $X$ be a mixed semi-isogenous surface with invariants $p_g(X)=q(X)=2$ and $K_X^2=7$.
\begin{enumerate}
\item $X$ is minimal, its minimal realization as mixed surface is given by $G^0\cong \ZZ/3\ZZ$ acting freely on a curve $C$ of genus $4$, $G \cong  \ZZ/6\ZZ$. 
\item The Albanese morphism of $X$  is a triple cover of $J(C/G^0)$ simply branched on $\rho_2(C/G^0)(\Delta_{C/G^0})$. More precisely, it  is the composition of $\pi_{G^0}$, finite triple cover of $\Sym^2(C/G^0)$ simply branched on $\Delta_{C/G^0}$, with $\rho_2(C/G^0)$. 
\end{enumerate}
\end{proposition}
\begin{proof}
\begin{enumerate}
\item \cite{CanFr15}*{Table 3}.
\item Applying Theorem \ref{Albanese and generalized dihedral surfaces} to this case we obtain $K=G^0$. Then the Albanese variety is $J(C/G^0)$ and the Albanese morphism is $\rho_2(C)  \circ \pi_{G^0}$.

By Proposition \ref{simply ramified}, since $\varphi$ is the identity, $\pi_{G^0}$ ramifies simply on the image of the graphs of the nontrivial elements of $G^0$. A graph of an element of $G^0$ dominate $\Delta_{C/G^0}$, that is then the branching locus of $\pi_{G^0}$. Since $\deg \pi_{G^0}=3$, simple ramification implies simple branching.
\end{enumerate}
\end{proof}
To describe the ramification divisor in the next cases we need a bit of geometry of a simple dihedral surface.
\begin{lemma}\label{doublecover}
Let $C$ be a curve of genus $3$ admitting a free involution $i$ and let $G^0=\langle i \rangle$, so $C/G^0$ is a curve of genus $2$. 
 
Consider the dihedral surface $X_{C,G^0}$. Then $\Alb(X_{C,G^0})$ is an Abelian surface with a polarization $(1,2)$, isogenous to $J(C/G^0)$ via an isogeny of degree $2$.

Consider the \'etale double cover $\pi_{G^0} \colon X_{C,G^0} \rightarrow \Sym^2(C/G^0)$.
Then $\pi_{G^0}^* \Delta_{C/G^0} = \Delta_1+\Delta_2$,  $\pi_{G^0}^* \Gamma_{C/G^0}= \Gamma_1+\Gamma_2$ with $\pi_{G^0|\Delta_i} \colon \Delta_i \rightarrow \Delta_{C/G^0}$, $\pi_{G^0|\Gamma_i} \colon \Gamma_i \rightarrow \Gamma_{C/G^0}$ isomorphisms. 

Moreover, $\Delta_1 \cap \Delta_2=\Gamma_1 \cap \Gamma_2=\emptyset$ and
 $\Delta_i$ and $\Gamma_j$ intersect transversally in $2(1 + \delta_{ij})$ points, where $\delta_{ij}$ is the usual Kronecker symbol.
\end{lemma}
\begin{proof}
Consider, in $C \times C$, the graph $\Gamma_{\Id}$ of the identity and  the graph $\Gamma_i$ of $i$. They are disjoint curves, both invariant for the actions of $G^0$ and $\sigma$ on $C \times C$ and then map on two disjoint curves $\Delta_1$ and $\Delta_2$ on $X_{C,G^0}$, both mapping on $\Delta_{C/G^0}$ when projected to $\Sym^2(C/G^0)$.  

Since the hyperelliptic involution of a curve of genus $2$ acts trivially on its $2-$torsion line bundles,  it lifts to every \'etale double cover. So the hyperelliptic involution of $C/G^0$ may be lifted to two different automorphisms  of $C$, say $\eta_1$ and $\eta_2$, forming with $\{ \Id, i \}$ a group of automorphisms of $C$ isomorphic to $\left(\ZZ/2\ZZ \right)^2$. Arguing as before, their graphs $\Gamma_{\eta_j}$ are $G^0-$invariant and map to two disjoint curves $\Gamma_j$ both in the preimage of $\Gamma_{C/G^0}$. 

An intersection point in $\Delta_i \cap \Gamma_j$ lies on $\Delta_{C/G^0} \cap \Gamma_{C/G^0}$. $\Delta_{C/G^0}$ and $\Gamma_{C/G^0}$ intersect transversally on the six points $p_k +p_k$ given by the six Weierstra\ss\ points $p_k$ of $C/G^0$.
Each of the $12$ preimages of them is fixed by exactly one $\eta_j$, and since  $\eta_1 \circ \eta_2=i$ acts freely, $\eta_1$ and $\eta_2$ have no common fixed point. By Hurwitz formula, up to change the labeling, $\eta_1$ and $\eta_2$ have respectively $8$ and $4$ fixed points. So $\Delta_C$ and $\Gamma_{\eta_1}$ intersect transversally in $8$ points and then their images $\Delta_1$ and $\Gamma_1$ intersect transversally in $4$ points. The other intersection numbers are similarly computed.
\end{proof}
 
 The Albanese morphism of $X_{C,G^0}$ is the contraction of the two disjoint rational curves $\Gamma_i$ to the two points kernel of the isogeny onto $J(C/G^0)$. The images of the $\Delta_i$ in $\Alb(X_{C,G^0})$ will be useful.

\begin{definition}\label{Delta'}
We set $\Delta'_i$ for the image  of $\Delta_i$ in $\Alb(X_{C,G^0})$.

These are two curves  with two singular points, an ordinary 4-ple point and an ordinary double point at the two points kernel of the isogeny $\Alb(X_{C,G^0})\rightarrow J(C/G^0)$, exchanged by the free involution of $\Alb(X_{C,G^0})$ induced by the isogeny. 

Since the action of this involution on the Neron-Severi group is trivial, $\Delta'_1$ and $\Delta'_2$  are numerically equivalent. More precisely, by \cite{pigpol}*{Lemma 2.7} the numerical class of $\Delta'_1+\Delta'_2$ is $4$ times the class of the polarization, and then the class of each $\Delta_i'$ is twice the class of the polarization. 
\end{definition}

Following \cite{CanFr15}*{Table 3}, the next cases are three families with $p_g=q=2$ and $K^2=6$. We separate them in two cases according to the isomorphism class of $G^0$.

We first treat the case when $G^0$ is not cyclic. Then $G^0$ is isomorphic to $\left( \ZZ/2\ZZ \right)^2$ and these surfaces are obtained by an \'etale $\left( \ZZ/2\ZZ \right)^2$-cover of a curve of genus $2$. 
It is wellknown that these covers form two families (see for example \cite{ortega}) and then they give two distinct families of mixed surfaces as well.

\begin{proposition}\label{G0=Z2xZ2}
Let $X$ be a mixed semi-isogenous surface with $p_g(X)=q(X)=2$, $K_X^2=6$ whose minimal realization has $G^0$ not cyclic.
\begin{enumerate}
\item There are two disjoint families of these surfaces. In both cases $X$ is  minimal, its minimal realization as mixed surface is given by $G^0\cong \left( \ZZ/2\ZZ \right)^2$ acting freely on a curve $C$ of genus $5$, $G \cong D(\ZZ/4\ZZ)$. 
\item The Albanese variety of $X$ is isomorphic to $\Alb(X_{C/Z(G),G^0/Z(G)})$, one of the Abelian varieties in Lemma \ref{doublecover}, isogenous to $J(C/G^0)$ via an isogeny of degree $2$.
The Albanese morphism of $X$ is a double cover branched on one of the curves $\Delta_i'$ in Definition \ref{Delta'}.
\end{enumerate}
\end{proposition}

\begin{proof}
\begin{enumerate}
\item \cite{CanFr15}*{Table 3}.
\item The group $K$ in  Theorem \ref{Albanese and generalized dihedral surfaces} in this case is the center $Z(G)$ of $G$. Then $C/K$ has genus $3$ and the further quotient $(C \times C)/G_K$ is then one of the dihedral surfaces of Lemma \ref{doublecover}.

So $\Alb(X)\cong\Alb(X_{C/Z(G),G^0/Z(G)})$ and the degree of the Albanese morphism of $X$ is  $2$. 

By Proposition \ref{ramification loci}   $\pi_{K}$ ramifies on the image of the graphs of the elements of $G^0$ not in $K$. Both map on $C/K \times C/K$ to the graph of a free involution $i$ of $C/K$ and then (see the beginning of the proof of Lemma \ref{doublecover}) the branching curve of $\pi_K$ is one of the $\Delta_i$ and the branching curve of the Albanese morphism is $\Delta_i'$.
\end{enumerate}

\end{proof}
\begin{corollary}\label{IaIb}
The two families of surfaces in Proposition \ref{G0=Z2xZ2} are a family of surfaces of type $Ia$ and a family of surfaces of type  $Ib$ in \cite{PePol13b}*{Theorem B}.

Surfaces of type $Ia$ and surfaces of type  $Ib$ in \cite{PePol13b}*{Theorem B} are topologically distinct.
\end{corollary}
\begin{proof}
The minimal surfaces of general type with $p_g=q=2$, $K^2=6$ and Albanese morphism of degree $2$ onto an Abelian surface with a polarization of type $(1,2)$ have been classified in \cite{PePol13b}. Their moduli space has three connected components, all irreducible, giving respectively the surfaces of type Ia, Ib and II. 

Since by \cite{CanFr15}*{Table 3} the surfaces in two different families of Proposition \ref{G0=Z2xZ2}
have different first homology group with integral coefficients, they are topologically distinct. Then our two families are contained in two different components of the moduli space. 

By \cite{PePol13b}*{Remark 20} the surfaces of type II have Albanese morphism with reducible branch locus. By Proposition \ref{G0=Z2xZ2}, since $\Delta'_i$ is irreducible, these two families are families of surfaces respectively of type Ia and Ib.
\end{proof}

The last family of surfaces with $K^2=6$ has $G^0$ cyclic of order $4$.

\begin{proposition}\label{G0=Z4}
Let $X$ be a mixed semi-isogenous surface with $p_g(X)=q(X)=2$, $K_X^2=6$ whose minimal realization has $G^0$ cyclic.
\begin{enumerate}
\item $X$ is minimal, its minimal realization as mixed surface is given by $G^0\cong \ZZ/4\ZZ$ acting freely on a curve $C$ of genus $5$, $G \cong  \ZZ/2\ZZ \times \ZZ/4\ZZ$. 
\item The Albanese variety of $X$ is isomorphic to $\Alb(X_{C/K,G^0/K})$ where $K$ is the subgroup of $G^0$ of order two. This is  one of the Abelian varieties in Lemma \ref{doublecover}, isogenous to $J(C/G^0)$ via an isogeny of degree $2$.
The Albanese morphism of $X$ is a double cover branched on one of the curves $\Delta_i'$ in Definition \ref{Delta'}.
\item These are either all  surfaces of type $Ia$ or all surfaces of type  $Ib$ in the classification given by \cite{PePol13b}*{Theorem B}.
 \end{enumerate}
\end{proposition}
\begin{proof} 
\begin{enumerate}
\item \cite{CanFr15}*{Table 3}.
\item Applying  Theorem \ref{Albanese and generalized dihedral surfaces} in this case we obtain that $K$ is the only subgroup of $G^0$ of order $2$. The rest of the proof is identical to the proof of Proposition \ref{G0=Z2xZ2}.
\item This follows by the argument of the proof of Corollary \ref{IaIb}.
\end{enumerate}
\end{proof}

We keep following \cite{CanFr15}*{Table 3}: the next case is the case $K_X^2=4$.

\begin{proposition}\label{CH}
Let $X$ be a mixed semi-isogenous surface with $p_g(X)=q(X)=2$, $K_X^2=4$.
\begin{enumerate}
\item The self-intersection of a canonical divisor of the minimal model of $X$ is $5$. $X$ has minimal realization given by $G^0 \cong D(\ZZ/3\ZZ)$, acting freely on a curve $C$ of genus $7$, $G \cong D(\ZZ/6\ZZ)$.
\item The Albanese variety of $X$ is isomorphic to $\Alb(X_{C/\left( \ZZ/3\ZZ \right),G^0/\left( \ZZ/3\ZZ \right) })$,  one of the Abelian varieties in Lemma \ref{doublecover}, isogenous to $J(C/G^0)$ via an isogeny of degree $2$.
The Albanese morphism of $X$ is a triple cover simply branched on one of the curves $\Delta_i'$ in Definition \ref{Delta'}.
\item These are all Chen-Hacon surfaces (see \cite{PePol13a}*{Definition 4.1}).
 \end{enumerate}
\end{proposition}

\begin{proof}
\begin{enumerate}
\item \cite{CanFr15}*{Table 3}.
\item Applying  Theorem \ref{Albanese and generalized dihedral surfaces} in this case we obtain $K=\ZZ/3\ZZ$, so the Albanese morphism has degree $3$. The rest of the proof is identical to the proof of Proposition \ref{G0=Z2xZ2}, using that a simply ramified triple cover is simply branched.
\item This follows by \cite{PePol13a}*{Theorem A}).
\end{enumerate}
\end{proof}

Finally,  \cite{CanFr15}*{Table 3} contains three families of surfaces with $K_X^2=2$ whose minimal models have  $K^2=4\chi({\mathcal O})=4$, so realizing the equality of Severi inequality. Indeed the description we obtain of their Albanese morphisms is coherent with the known characterization of these surfaces in \cite{BPS}, \cite{Zhang}.

We need to consider a family of dihedral surfaces different from the one in Lemma \ref{doublecover}.

\begin{lemma}\label{bidoublecover}
Let $C$ be a curve of genus $5$ admitting a free action of $G^0 \cong ( \ZZ/2\ZZ )^2$, so $C/G^0$ is a curve of genus $2$. 
 
Consider the dihedral surface $X_{C,G^0}$. Then $\Alb(X_{C,G^0})$ is a principally polarized Abelian surface  isogenous to $J(C/G^0)$ via an isogeny with kernel isomorphic to $G^0$.

Consider the \'etale bidouble cover $\pi_{G^0} \colon X_{C,G^0} \rightarrow \Sym^2(C/G^0)$. Then $\pi_{G^0}^*( \Delta_{C/G^0})=\sum_{g \in G^0} \Delta_g$ and   $\pi_{G^0}^* (\Gamma_{C/G^0})=\sum_{g \in G^0} \Gamma_g$ have both $4$ connected components such that  all  maps $\pi_{G^0|\Delta_g} \colon \Delta_g \rightarrow \Delta_{C/G^0}$ and 
$\pi_{G^0|\Gamma_g} \colon \Gamma_g \rightarrow \Gamma_{C/G^0}$ are isomorphisms. Moreover
\begin{itemize}
\item[i)] either $\Delta_g \Gamma_{g}= 0$ for all $g$ ,  whereas $\Delta_g \Gamma_{g'}= 2$ for all pair of distinct elements $g,g'$,
\item[ii)] or $\Delta_g \Gamma_{g}= 3$ for all $g$ ,  whereas $\Delta_g \Gamma_{g'}= 1$ for all pair of distinct elements $g,g'$.
\end{itemize}

\end{lemma}
\begin{proof}
The proof  is almost identical to the proof of Lemma \ref{doublecover}. We sketch then here only the computation of the fixed points of the $4$ lifts of $\eta_{C/G^0}$ to $C$, producing the dichotomy in the final statement.

The quotient map $C \rightarrow C/G^0$ is the fiber product of two double covers as in Lemma \ref{doublecover}, so we have two lifts of the hyperelliptic involution $\eta_{C/G^0}$ to each double cover. 

Denote by $\eta_1$ and $\eta_2$ the two lifts to the first double cover and by $\eta'_1$, $\eta'_2$ the two lifts to the second double cover. By the proof of Lemma \ref{doublecover} we can assume that $\eta_2$ fixes $4$ points, preimage of two Weierstra\ss\ points $p_1,p_2$ of $C/G^0$ (so  $\eta_1$ fixes the preimages of the other $4$  Weierstra\ss\ points) and similarly  $\eta'_2$ fixes  the preimages of two Weierstra\ss\ points $p'_1,p_2'$ .

We write the $4$ lifts of $\eta_{C/G^0}$ to $C$ as ordered pairs $(\eta_i,\eta'_j)$. If $\{p_1,p_2\}=\{p'_1,p'_2\}$, then  $\eta_1$ and $\eta'_1$ would both fix the same $8$ points and then $(\eta_1,\eta_1')$ would fix $16$ points of $C$, contradicting Hurwitz formula. 

Then the cardinality of  $\{p_1,p_2\} \cap \{p'_1,p'_2\}$ is either $0$ or $1$.
It is now easy to compute the intersection number $\Delta_g\Gamma_{g'}$ in the two cases, obtaining the dichotomy in the statement.
\end{proof}
Both the cases i) and ii) in Lemma \ref{bidoublecover} occur, giving the two already mentioned families of \'etale bidouble covers studied recently in \cite{ortega}.

The Albanese morphism of $X_{C,G^0}$ is the contraction of the curves $\Gamma_g$, mapped to the kernel of this isogeny. 

\begin{definition}\label{Delta''}
The images of the $\Delta_g$ in the abelian variety  $\Alb(X_{C,G^0})$ in Lemma \ref{bidoublecover} are four curves $\Delta''_g$ such that
\begin{itemize}
\item[i] either the singularities of $\Delta''_g$ are three ordinary double points  at three of the four points of the kernel of the isogeny $\Alb(X_{C,G^0})\rightarrow J(C/G^0)$, and  $\Delta''_g$ does not contains the fourth,
\item[ii] or $\Delta''_g$ has an ordinary triple point at one of the points of the kernel as only singularity, and contains the other three.
\end{itemize}
\end{definition}

\begin{proposition}\label{K^22}
Let $X$ be a mixed semi-isogenous surface with $p_g(X)=q(X)=2$, $K_X^2=2$.
\begin{enumerate}
\item The self-intersection of a canonical divisor of the minimal model of $X$ is $4$.  $X$ has minimal realization given by a group $G^0$ of order $8$ acting freely on a curve of genus $9$. More precisely 
\subitem either $G^0\cong D(\ZZ/4\ZZ)$, $G \cong  D(\ZZ/4\ZZ) \times \ZZ/2\ZZ$,
\subitem or $G^0=\langle i,j,k| i^2=j^2=k^2=ijk=-1\rangle$ is the group of quaternions and $G$ is the central product of $G^0$ with $\ZZ/4\ZZ$ over a common cyclic central subgroup of order $2$.
\item $\Alb(X)\cong\Alb(X_{C/Z(G^0),G^0/Z(G^0)})$ is one of the Abelian varieties in Lemma \ref{bidoublecover}, isogenous to $J(C/G^0)$ with the kernel of the isogeny isomorphic to $\left( \ZZ/2\ZZ \right)^2$.
The Albanese morphism of $X$ is a double cover branched on one of the curves $\Delta_g''$ in Definition \ref{Delta''}.
 \end{enumerate}
\end{proposition}

\begin{proof}
\begin{enumerate}
\item \cite{CanFr15}*{Table 3}.
\item Theorem \ref{Albanese and generalized dihedral surfaces} gives in both cases  $K=Z(G^0)\cong \ZZ/2\ZZ$,  $G^0/Z(G^0) \cong  \left( \ZZ/2\ZZ \right)^2$ and then the Albanese morphisms is a double cover of one of the Abelian varieties in Lemma \ref{bidoublecover}.

By Proposition \ref{ramification loci}   $\pi_{K}$ ramifies in both cases, on two curves $\Gamma_g$ and $\Gamma_{g^{-1}}$, mapping  to the same curve in $X_{C/Z(G^0),G^0/Z(G^0)}$. So the branching curve of the Albanese morphism is irreducible and dominates $\Delta_{C/G^0} \subset J(C/G^0)$, so it is one of the curves in Definition \ref{Delta''}.
\end{enumerate}
\end{proof}


\end{document}